\documentclass[a4paper, 12pt]{amsart}

\makeatletter

\@addtoreset{equation}{section}
\makeatother 

\theoremstyle{plain}
\newtheorem{theorem}{Theorem}[section]
\newtheorem{lemma}[theorem]{Lemma}
\newtheorem{corollary}[theorem]{Corollary}
 
\newtheorem*{theorem*}{Theorem} 

\theoremstyle{definition}

\theoremstyle{remark}
\newtheorem{remark}[theorem]{Remark} 
\newtheorem*{acknowledgements}{\bf{Acknowledgments}} 

\newcommand{\R}{\mathbb{R}}

\renewcommand{\a}{\alpha}

\newcommand{\e}{\varepsilon}

\newcommand{\na}{\nabla}

\newcommand{\p}{\partial}
\newcommand{\ddt}{\frac{d}{dt}}
\newcommand{\ddr}{\frac{d}{dr}}

\newcommand{\abs}[1]{\left\lvert#1\right\rvert}

\DeclareMathOperator{\supp}{supp}

\newcommand{\dm}{d\mu} 
\newcommand{\dt}{dt} 
\newcommand{\dr}{dr} 

\newcommand{\dy}{dy}

\title[Monotone Integral Quantities of Mean Curvature Flow]{On Ecker's local integral quantity at infinity for ancient mean curvature flows}

\author[]{Keita Kunikawa} 
\address{Advanced Institute for Materials Research (AIMR), Tohoku University, 2-1-1 Katahira, Aoba-ku, Sendai, 980-8577, Japan}
\email{keita.kunikawa.e2@tohoku.ac.jp} 
 
\thanks{The author
was supported by JSPS KAKENHI Grant Number JP19K14521.}
\date{\today}

\begin{document}
\maketitle

\begin{abstract}
We point out that Ecker's local integral quantity agrees with Huisken's global integral quantity at infinity for ancient mean curvature flows if Huisken's one is finite on each time-slice. In particular, this means that the finiteness of Ecker's integral quantity at infinity implies the finiteness of the entropy at infinity. 
\end{abstract}

\section{Introduction and Results} 
There are two monotone integral quantities along mean curvature flow. One is Huisken's (global) integral quantity \cite{Hu90}, and the other is Ecker's (local) integral quantity \cite{Ec01}. It is known by Ecker that these two integral quantities coincide with each other as a limit at zero. 
The purpose of this paper is to show a relation between them at infinity for ancient mean curvature flows. More precisely, we prove that Ecker's integral quantity coincides with Huisken's integral quantity at infinity if Huisken's one is finite on each time-slice (see Theorem \ref{thm1}). This result allows us to interchange some arguments about the monotone quantities at infinity. Especially, we can replace the entropy at infinity by Ecker's integral quantity at infinity (see Corollaries \ref{dimbd} and \ref{cor3.2}). 

First, we recall some definitions, notations and previous results on the monotone quantities in the next four subsections. 
\subsection{Mean curvature flow}
Let $M^n$ be an $n$-dimensional smooth manifold (not necessarily to be compact) and $x_t=x(\cdot, t): M\times I \to \R^{N}$ be a one-parameter family of smooth immersions, where $I\subset \R$ is an interval. We consider the \textit{mean curvature flow} 
\[\frac{\p x}{\p t} = H(x) \]
with images $M_t=x_t(M)$, where $H$ denotes the mean curvature vector field of $x_t$.  
Denote by  
\[\mathcal{M}:=\bigcup_{t\in I}M_t\times \{t\}\subset \R^{N}\times \R \]
the \textit{space-time track} of mean curvature flow. We also write $\mathcal{M}=(M_t)_{t\in I}$ for an abbreviation. Each $M_t$ is called a \textit{time-slice}. We always assume that each time-slice $M_t$ has no boundary in $\R^N$. We say that the mean curvature flow is \textit{ancient} if $I=(-\infty, 0)$. Static solutions (i.e., minimal submanifolds), self-shrinking solutions and translating solutions are typical examples of ancient mean curvature flows.  

\subsection{Huisken's monotonicity formula}
In \cite{Hu90} Huisken introduced the following integral quantity 
\[\int_{M_t}\Phi \dm_t, \quad \text{where }\ \Phi(x,t) := \frac{1}{(-4\pi t)^{\frac{n}{2}}}e^{\frac{|x|^2}{4t}} \quad (t<0), \] 
and obtained the \textit{monotonicity formula} 
\begin{align}
	\ddt\int_{M_t}\Phi \dm_t=-\int_{M_t}\abs{H-\frac{x^\perp}{2t}}^2\Phi\dm_t=-\int_{M_t}|H-\na^\perp\psi|^2\Phi\dm_t  \leq 0 \label{Huisken}
\end{align} 
as long as $\int_{M_t}\Phi \dm_t<\infty $ for each $t<0$. Here we have put 
\[\psi_r:=\log(\Phi r^n):=\psi+n\log r=\frac{|x|^2}{4t}-\frac{n}{2}\log\left(\frac{-4\pi t}{r^2}\right). \] 
in \eqref{Huisken}. 
\begin{remark} \label{rmk_mono}
	Choose any point $(x_0, t_0)\in \R^N\times \R$ and fix it. We define 
	\[\Phi_{x_0, t_0}(x, t):=\Phi(x-x_0, t-t_0)\] 
	for $x\in \R^N$ and $t<t_0$. Then Huisken's monotonicity formula with respect to $(x_0, t_0)$ becomes  
	\[\ddt\int_{M_t}\Phi_{x_0, t_0}\dm_t=-\int_{M_t}\abs{H-\frac{(x-x_0)^\perp}{2(t-t_0)}}^2\Phi_{x_0, t_0}\dm_t \]
	accordingly. 
\end{remark} 

By Huisken's monotonicity formula, we can take the limit: 
\begin{align*}
	\Theta(\mathcal{M}, x_0, t_0):=\lim_{t\to t_0}\int_{M_t}\Phi_{x_0, t_0}\dm_t.
\end{align*}
Also the following limit exist (being allowed to become infinity):  
\begin{align*}
	\Theta(\mathcal{M}, \infty):=\lim_{t\to -\infty}\int_{M_t}\Phi_{x_0, t_0}\dm_t. 
\end{align*} 
Note that $\Theta(\mathcal{M}, \infty)$ does not depend on the choice of $(x_0, t_0)\in \R^N\times \R$.

\subsection{Ecker's monotonicity formula} 
In \cite{Ec01}, on the other hand, Ecker introduced another local integral quantity 
\begin{align*}
	\mathcal{A}(\mathcal{M}\cap E_r)=\iint_{\mathcal{M}\cap E_r}\abs{\na\psi}^2+|H|^2\psi_r \quad (r>0), 
\end{align*} 
where $E_r$ is the so-called \textit{heat-ball} 
\[E_r:=\{(x, t)\in \R^{N}\times (-\infty, 0)\mid \psi_r > 0\}\subset \R^{N}\times \R. \] 
Note that $\psi_r=0$ on the boundary $\partial E_r$ by definition. The heat-ball $E_r$ can be  written as the form 
\[E_r=\bigcup_{-\frac{r^2}{4\pi}<t<0}B_{R_r(t)}\times \{t\}, \quad \text{where } \ R_r(t):=\sqrt{2nt\log\left(\frac{-4\pi t}{r^2}\right)}. \]
Therefore, for ancient mean curvature flow $\mathcal{M}$, the integral of a function $f$ on a heat-ball $E_r$ means  
\[\iint_{\mathcal{M}\cap E_r}f=\int_{-\frac{r^2}{4\pi}}^0\int_{M_t\cap B_{R_r(t)}}f \dm_t \dt. \]
It is known that Ecker's integral behaves like $n$-dimensional volume. For example, it 
scales like 
\[\frac{\mathcal{A}(\mathcal{M}\cap E_r)}{r^n}=\mathcal{A}(\mathcal{M}^{r}\cap E_1) \quad \text{for all } r>0,\]
where $\mathcal{M}^{r}=(r^{-1}M_{r^2 t})_{t<0}$ is the so-called \textit{parabolic rescaling} of $\mathcal{M}$. In addition, Remark 3.2 in \cite{Ec01} leads to the estimate 
	\begin{align} \label{Ec-vs-Hui}
		\frac{\mathcal{A}(\mathcal{M}\cap E_r)}{r^n}\leq \frac{c_n\mathcal{H}^n\Big(M_{-\frac{r^2}{4\pi}}\cap B_{\sqrt{\frac{2n}{\pi}}r}\Big)}{r^n}\leq e^{2n}c_n\int_{M_{-\frac{r^2}{4}}}\Phi\dm_{-\frac{r^2}{4}}   
	\end{align}
for all $r>0$, where $\mathcal{H}^n$ denotes the $n$-dimensional Hausdorff measure and $c_n$ is a dimensional constant. Moreover, on self-shrinking solutions, Ecker's integral exactly agrees with Huisken's one (Proposition 3.3 in \cite{Ec01}): for all $r>0$ and all $t<0$, we have 
\[\frac{\mathcal{A}(\mathcal{M}\cap E_r)}{r^n}=\int_{M_t}\Phi \dm_t. \] 

We say that $\mathcal{M}$ is \textit{well-defined} (see \cite{Ec01}, \cite{Ec04}) in $B_{\sqrt{\frac{2n}{\pi}}r}\times (-\frac{r^2}{4\pi}, 0)$ if it has no boundary inside this set and satisfies the condition 
		\begin{align*}
			\mathcal{H}^n\left(M_{-\frac{r^2}{4\pi}}\cap B_{\sqrt{\frac{2n}{\pi}}r}\right)<\infty. 
		\end{align*} 
In particular, when $\mathcal{M}$ is ancient, we say that $\mathcal{M}$ is \textit{well-defined} (see \cite{Ca06}) in $\R^N\times (-\infty, 0)$ if each $M_t$ has no boundary in $\R^N$ and has locally finite mass 
		\begin{align*}
			\mathcal{H}^n\left(M_t \cap B_{2\sqrt{-2nt}}\right)<\infty.  
		\end{align*} 
Well-definedness of $\mathcal{M}$ guarantees that all integral quantities considered in this paper are locally finite. 

In \cite{Ec01} (Theorem 3.4), Ecker derived the \textit{local monotonicity formula} 
\begin{align} \label{localmono}
	\ddr\left( \frac{\mathcal{A}(\mathcal{M}\cap E_r)}{r^n}\right)=\frac{n}{r^{n+1}}\iint_{\mathcal{M}\cap E_r} |H-\na^\perp \psi|^2\geq 0      	
\end{align} 
for well-defined $\mathcal{M}$. Using \eqref{localmono}, he showed that the limit of the integral quantity at zero satisfies   
\begin{align}
	\lim_{r\to 0}\frac{\mathcal{A}(\mathcal{M}\cap E_r)}{r^n}=\Theta(\mathcal{M}, 0, 0)=\lim_{t\to 0}\int_{M_t}\Phi \dm_t. \label{at0} 
\end{align}
In this sense, Ecker's integral quantity can be interpreted as a localization of Huisken's global one. 
\begin{remark} \label{Ec<Hui}
	By the estimate \eqref{Ec-vs-Hui}, Ecker's integral quantity is always bounded from above by Huisken's one. We know a better bound for this. In fact, integrating the local monotonicity formula and comparing with Huisken's one, 
		\begin{align*}
			\frac{\mathcal{A}(\mathcal{M}\cap E_r)}{r^n}\leq \int_{M_{-\frac{r^2}{4}}}\Phi\dm_{-\frac{r^2}{4}}
		\end{align*}
holds for all $r>0$ in general.  
\end{remark}

\subsection{Ecker's integral vs. Huisken's integral at infinity} \label{subsection EH}
Observing above facts, we come up with a question about the relation between these integrals at infinity. Inspired by the work of Yokota \cite{Yo10} for ancient solution of Ricci flow, we expect that the above two monotone quantities coincide with each other at infinity. In fact, this is true if we assume 
	\[\Theta(\mathcal{M}, \infty):=\lim_{t\to -\infty}\int_{M_t}\Phi\dm_t\left(=\sup_{t<0}\int_{M_t}\Phi\dm_t\right)<\infty. \] 
	Here we will give a sketch of the proof under the assumption $\Theta(\mathcal{M}, \infty)<\infty$. First, integrating the local monotonicity formula, we know the following expression by Ecker (in \cite{Ec01}, Remark 3.5): for all $r>0$, 
	\begin{align}
	\frac{\mathcal{A}(\mathcal{M}\cap E_r)}{r^n}-\Theta(\mathcal{M}, 0, 0)=\iint_{\mathcal{M}\cap E_r}|H-\na^\perp\psi|^2\Phi-\frac{1}{r^n}\iint_{\mathcal{M}\cap E_r}|H-\na^\perp\psi|^2. \label{ER35}
	\end{align} 
	Since we assume $\Theta(\mathcal{M}, \infty)<\infty$, Huisken's monotonicity formula combined with the works of Ilmanen \cite{Il93}, \cite{Il95} and White \cite{Wh94} shows that every sequence $r_i\to \infty$ has a subsequence (also denoted by $r_i$) such that $\mathcal{M}^{r_i}$ converges to a self-shrinker $\mathcal{M}^\infty$ (this limit is called a  \textit{tangent flow at $-\infty$} or a \textit{blow-down limit}). Thus, taking such a subsequence $r_i$ in \eqref{ER35}, we have 
	\begin{align} \label{blowdown}
	\lim_{r_i\to\infty}\frac{\mathcal{A}(\mathcal{M}\cap E_{r_i})}{r_i^n}-\Theta(\mathcal{M}, 0, 0)=\lim_{r_i\to\infty}\iint_{\mathcal{M}\cap E_{r_i}}|H-\na^\perp\psi|^2\Phi, 
	\end{align}
	but this equality actually holds without extracting subsequence since the quantities in both sides are monotone. 
	
	On the other hand, integrating Huisken's monotonicity formula, one obtains  
	\begin{align}
		\Theta(\mathcal{M}, \infty)-\Theta(\mathcal{M}, 0, 0)=\iint_{\mathcal{M}} |H-\na^\perp\psi|^2\Phi. \label{Hui-limit}
	\end{align} 
	Comparing \eqref{blowdown} and \eqref{Hui-limit}, we conclude that 
	\begin{align*}
		\lim_{r\to \infty}\frac{\mathcal{A}(\mathcal{M}\cap E_r)}{r^n}= \Theta(\mathcal{M},\infty). 
	\end{align*}
This completes the proof with the condition $\Theta(\mathcal{M}, \infty)<\infty$. 

\subsection{Main results} 
In the current paper, we show that the same conclusion holds as in Subsection \ref{subsection EH} under a mild assumption. The point is that we do not need the uniform bound $\Theta(\mathcal{M}, \infty) < \infty$. Instead, we only assume the finiteness of Huisken's integral on each time-slice.

\begin{theorem} \label{thm1}
Let $\mathcal{M}=(M_t)_{t<0}$ be an ancient mean curvature flow which satisfies 
\begin{align}
	\int_{M_t}\Phi \dm_t <\infty \quad  \text{for each } \ t<0. \label{finite}
\end{align} 
 Then we have 
	\begin{align}
		\lim_{r\to \infty}\frac{\mathcal{A}(\mathcal{M}\cap E_r)}{r^n}= \Theta(\mathcal{M},\infty). \label{main}
	\end{align}
\end{theorem} 

\begin{remark} (a) The both sides in \eqref{main} are allowed to become infinity. 

(b) In view of \eqref{Ec-vs-Hui}, the condition \eqref{finite} implies that each $M_t$ has locally finite $\mathcal{H}^n$-measure. In particular, $\mathcal{M}$ is well-defined. It is known that \eqref{finite} is satisfied when each $M_t$ has at most polynomial volume growth (see for example Remark 4.12 in \cite{Ec04}). 

(c) We do not known whether there exists an example where Huisken's integral is infinite on each time-slice but the limit at infinity of Ecker's integral is finite. Such an example (if exists) cannot have polynomial volume growth on each time-slice by (b). This means that each time-slice is geometrically complex at infinity in space. On the other hand, its blow-down limit must be relatively simple due to the finiteness of Ecker's integral at infinity. So it seems hard to find such an example. 
\end{remark}

As we will see later (Lemma \ref{ent} below), one can show that 
\[\Theta(\mathcal{M}, \infty)=\sup_{t<0}\lambda(M_t), \]
where $\lambda(M_t)$ denotes the \textit{entropy} of $M_t\subset \R^N$, which is monotone nonincreasing along mean curvature flow (see Section 3). Thus, as an application of Theorem \ref{thm1}, we have   
\begin{align} \label{Ecker=entropy}
	\lim_{r\to \infty}\frac{\mathcal{A}(\mathcal{M}\cap E_r)}{r^n}=\sup_{t<0}\lambda(M_t)  
\end{align} 
under the assumption \eqref{finite}. Recently, Colding-Minicozzi \cite{CM19} obtained some codimension bound by entropy under the assumption that $\sup_{t<0}\lambda(M_t)<\infty$. This kind of codimension bound was initiated by Calle \cite{Ca06} using Ecker's integral quantity. 
\begin{theorem}[Colding-Minicozzi \cite{CM19}, Corollary 0.6] \label{CM bound}
There exists a dimensional constant $c_n$ so that if $\mathcal{M}=(M_t)_{t<0}$ is an ancient mean curvature flow with  $\sup_{t<0}\lambda(M_t)<\infty$, then it is contained in a Euclidean subspace of dimension $\leq c_n \sup_{t<0}\lambda(M_t)$. 
\end{theorem} 

Using the relation \eqref{Ecker=entropy}, the assumption on entropy in Theorem \ref{CM bound} can be replaced by the assumption on Ecker's integral quantity (see also Calle \cite{Ca06}). 
\begin{corollary} \label{dimbd}
	There exists a dimensional constant $c_n$ so that if $\mathcal{M}=(M_t)_{t<0}$ is an ancient mean curvature flow with \eqref{finite} and $\lim_{r\to\infty}r^{-n}\mathcal{A}(\mathcal{M}\cap E_r)<\infty$, then it is contained in a Euclidean subspace of dimension $\leq c_n \sup_{t<0}\lambda(M_t) = c_n\lim_{r\to\infty}r^{-n}\mathcal{A}(\mathcal{M}\cap E_r)$. 
\end{corollary}
\begin{remark}
	 It follows from \eqref{Ec-vs-Hui} or Remark \ref{Ec<Hui} that $\sup_{t<0}\lambda(M_t)<\infty$ implies  
	\[\int_{M_t}\Phi\dm_t<\infty \quad \text{for each $t<0$}, \quad \text{and} \quad \lim_{r\to \infty}\frac{\mathcal{A}(\mathcal{M}\cap E_r)}{r^n}<\infty. \] 
	We emphasize that the converse of this fact is nontrivial. Theorem \ref{thm1} combining with \eqref{Ecker=entropy} says that condition \eqref{finite} and the finiteness of Ecker's integral quantity at infinity lead to the finiteness of entropy at infinity. Then applying the codimension bound by Colding-Minicozzi, we obtain Corollary \ref{dimbd}. 
\end{remark}

\begin{acknowledgements}
	The author would like to thank Yohei Sakurai for helpful discussions during this work. The author is grateful to Takumi Yokota for giving him a rough idea of the proof in \cite{Yo10} for Ricci flow. 
\end{acknowledgements}

\section{Proof of Theorem \ref{thm1}} 
In this section, we give the proof of Theorem \ref{thm1}. We follow the same line as in \cite{EKNT08} and \cite{Yo10} (see also \cite{Ec01}). 
The proof is divided into several steps. In the first half of the proof, we repeat the proof of Ecker's monotonicity which is already known by the argument in \cite{Ec01} or \cite{EKNT08}. However, this step is unavoidable since we use some expression for the monotonicity of approximated Ecker's integral in our proof (see \eqref{mono} below). 

In the proof, we always assume the condition \eqref{finite}, that is, Huisken's integral is finite for each time-slice. 

\subsection{Smooth approximation of Ecker's integral}
First, we recall that the condition \eqref{finite} ensures the well-definedness of $\mathcal{M}$, so Ecker's integral quantity $\mathcal{A}(\mathcal{M}\cap E_r)$ makes sense for every $r>0$. In order to make the behavior of $\psi_r$ near $t=0$ clear, we take a truncated subset of $\mathcal{M}$:   
	\[\mathcal{M}_s:=\bigcup_{t\in (-\infty, s]} M_t\times \{t\}\subset \mathcal{M}, \quad -\infty<s<0. \]
 For $r>0$, as a restriction of $\mathcal{A}(\mathcal{M}\cap E_r)$ to the time interval $(-\infty, s]$, we consider an integral quantity 
	\begin{align*}
		\mathcal{A}(\mathcal{M}_s\cap E_r):=\iint_{\mathcal{M}_s\cap E_r}|\na\psi|^2+|H|^2\psi_r. 
	\end{align*} 
Moreover, we consider a smooth approximation of $\mathcal{A}(\mathcal{M}\cap E_r)$ in the following. The smooth approximation is needed to avoid difficulties arising from the fact that we have no control on the regularity of the heat-ball $E_r$. Fix small $\e>0$ and take a smooth function $\eta_\e:\R\to \R_+$ which satisfies the following properties: 
	\begin{itemize}
		\item $\supp \eta_\e$ is contained in $[0, \e]$,  
		\item $\int_{-\infty}^{\infty}\eta_\e(y)\dy=1$. 
	\end{itemize}
	Then we define $\zeta_\e:\R\to \R_+$ and $Z_\e:\R\to \R_+$ by 
	\[\zeta_\e(x):=\int_{-\infty}^x\eta_\e(y)\dy, \quad Z_\e(x):=\int_{-\infty}^x\zeta_\e(y)\dy. \] 
	Clearly $\eta_\e(x)$ is a smooth approximation of the Delta function $\delta(x)$. Hence it is easy to see that 
	\begin{align} \label{zeta}
		\chi(x-\e)\leq \zeta_\e(x)\leq \chi(x), \quad \text{for all } x\in \R,
	\end{align}
	where $\chi$ is the Heaviside function. Integrating these inequalities, we have 
	\begin{align} \label{Z}
		[x-\e]_+\leq Z_\e(x)\leq [x]_+, \quad \text{for all } x\in \R, 
	\end{align}
	where $[x]_+=\max\{x, 0\}$. Therefore we can say that 
	\begin{itemize}
		\item $\zeta_\e$ approaches to the Heaviside function $\chi$ as $\e \to 0$, 
		\item $Z_\e$ approaches to the function $x\mapsto [x]_+=\max\{x, 0\}$ as $\e \to 0$. 
	\end{itemize} 
	Now the desired smooth approximation is defined by 
	\[A_\e(s, r):=\iint_{\mathcal{M}_s}|\na\psi|^2\zeta_\e(\psi_r)+|H|^2Z_\e(\psi_r). \] 
		Note that $\zeta_\e(\psi_r)$ and $Z_\e(\psi_r)$ vanish outside of $ E_r$ since $\psi_r\leq 0$. Hence $A_{\e}(s,r)$ and its limit 
	\[A_\e(r):=\lim_{s\to 0}A_\e(s, r) \]
	make sense. 		
\subsection{Derivative of the integral}
	
	Next we compute 
	\begin{align}
		\ddr \left(\frac{A_\e(s, r)}{r^n}\right)&=
		\frac{n}{r^{n+1}}\left(\frac{r}{n}\ddr A_\e(s, r)-A_\e(s, r)\right) \label{ddr1} \\
		&=\frac{n}{r^{n+1}}\iint_{\mathcal{M}_s}|\na\psi|^2\zeta'_\e+|H|^2\zeta_\e-|\na\psi|^2\zeta_\e-|H|^2Z_\e. \nonumber
	\end{align}
	On the other hand, we have 
	\begin{align}
		\ddt\int_{M_t}Z_\e(\psi_r)\dm_t = \int_{M_t} \zeta_\e\frac{d\psi}{dt}-Z_\e|H|^2 \dm_t, \label{ddt1}
	\end{align}
	where we have used the well known fact that 
	\[\frac{\p}{\p t}\dm_t=-|H|^2\dm_t\]
	along mean curvature flow. Integrating \eqref{ddt1} in $\tau\leq t\leq s$, we obtain 
	\begin{align*}
		\int_{M_{s}}Z_\e(\psi_r)\dm_{s}-\int_{M_{\tau}}Z_\e(\psi_r)\dm_{\tau}=\int_{\tau}^{s} \dt \int_{M_t} \zeta_\e\frac{d\psi}{dt}-Z_\e|H|^2 \dm_t. 
	\end{align*}
	It follows by taking a limit $\tau \to -\infty$ that  
	\begin{align}
		\int_{M_s}Z_\e(\psi_r)\dm_s=\iint_{\mathcal{M}_s} \zeta_\e\frac{d\psi}{dt}-Z_\e|H|^2 \label{s0infty}  
	\end{align} 
	since $\zeta_\e(\psi_r)=Z_\e(\psi_r)=0$ outside of $E_r$. Insert \eqref{s0infty} into \eqref{ddr1} to obtain 
	\begin{align}
		\ddr \left(\frac{A_\e(s, r)}{r^n}\right)=\frac{n}{r^{n+1}}\left[\iint_{\mathcal{M}_s}|\na\psi|^2\zeta'_\e+|H|^2\zeta_\e-|\na\psi|^2\zeta_\e-\frac{d\psi}{dt}\zeta_\e +\int_{M_s}Z_\e(\psi_r)\dm_s\right]. \label{ddr2}
	\end{align}
	Using   
	\[|\na\psi|^2\zeta'_\e=\langle \na\psi, \zeta'_\e\na\psi\rangle =\langle \na\psi, \na\zeta_\e \rangle \] 
	and $\zeta_\e(\psi_r)=0$ on $\p B_{R_r(t)}$, \eqref{ddr2} becomes	\begin{align*}
		\ddr \left(\frac{A_\e(s, r)}{r^n}\right)&=\frac{n}{r^{n+1}}\left[\iint_{\mathcal{M}_s}\langle \na\psi, \na\zeta_\e \rangle+|H|^2\zeta_\e-|\na\psi|^2\zeta_\e-\frac{d\psi}{dt}\zeta_\e +\int_{M_s}Z_\e(\psi_r)\dm_s\right] \\
		&= \frac{n}{r^{n+1}}\left[\iint_{\mathcal{M}_s}\left(-\Delta\psi+|H|^2-|\na\psi|^2-\frac{d\psi}{dt}\right)\zeta_\e +\int_{M_s}Z_\e(\psi_r)\dm_s\right]. 
	\end{align*}
	Now we need the computation from \cite{Ec01}. 
	\begin{lemma}[\cite{Ec01}, Lemma 3.1] \label{HElem}
	Along mean curvature flow, we have 
		\[\left(\ddt+\Delta \right)\psi=-|\na\psi|^2-|H-\na^\perp\psi|^2+|H|^2. \] 
	\end{lemma}
    Using Lemma \ref{HElem}, we get 
	\begin{align}
		\ddr \left(\frac{A_\e(s, r)}{r^n}\right)=\frac{n}{r^{n+1}}\left[\iint_{\mathcal{M}_s}|H-\na^\perp\psi|^2\zeta_\e(\psi_r)+\int_{M_s}Z_\e(\psi_r)\dm_s\right]. \label{ddr3}
	\end{align} 
	
\subsection{Error term estimate and the monotonicity of Ecker's integral} 
	In this subsection, we will show the monotonicity of $r^{-n}\mathcal{A}(\mathcal{M}\cap E_r)$. To do so, we integrate \eqref{ddr3} with respect to $r$ in $0<\sigma\leq r\leq \rho <\infty$ and apply Fubini's theorem. Then we have 
	\begin{align*}
		\frac{A_\e(s, \rho)}{\rho^n}-\frac{A_\e(s, \sigma)}{\sigma^n}&=\int_{\sigma}^{\rho}\dr  \iint_{\mathcal{M}_s}\frac{n}{r^{n+1}}|H-\na^\perp\psi|^2\zeta_\e(\psi_r)+E(s;\sigma, \rho)\\ 
		&=\int_{-\infty}^s\dt \int_{\sigma}^{\rho}\dr \int_{M_t}\frac{n}{r^{n+1}}|H-\na^\perp\psi|^2\zeta_\e(\psi_r)\dm_t +E(s;\sigma, \rho),  
	\end{align*}
	where 
	\[E(s;\sigma, \rho):=\int_{\sigma}^{\rho}\dr \int_{M_s}\frac{n}{r^{n+1}}Z_\e(\psi_r)\dm_s \]
	is an error term. 
	\begin{lemma}[Error term estimate] \label{error}
		\[\lim_{s\to 0}E(s; \sigma, \rho)=0. \]
	\end{lemma}
	\begin{proof} 
		By the volume bound derived in \cite{Ec01} (Lemma 1.2) and well-definedness of $\mathcal{M}$ (this follows from \eqref{finite}), we compute 
	\begin{align*}
		\limsup_{s\to 0} E(s;\sigma, \rho) &\leq \limsup_{s\to 0} \frac{n(\rho-\sigma)}{\sigma^{n+1}}\int_{M_s}Z_\e(\psi_{\rho})\dm_s \\ 
		&\leq \limsup_{s\to 0}\frac{n^2(\rho-\sigma)}{2\sigma^{n+1}}\log\left(\frac{\rho^2}{-4\pi s}\right)\mathcal{H}^n\left(M_s\cap B_{R_{\rho}(s)}\right) \\
		&\leq \limsup_{s\to 0}\frac{n^2(\rho-\sigma)}{2\sigma^{n+1}}\log\left(\frac{\rho^2}{-4\pi s}\right)c(n)\mathcal{H}^n\left(M_{-\frac{\rho^2}{4\pi}}\cap B_{\sqrt{\frac{2n}{\pi}}\rho}\right)\frac{R_{\rho}(s)^n}{\rho^n} \\
		&=0.   
	\end{align*}
	\end{proof}
	Letting $s\to 0$ and using Lemma \ref{error}, we obtain
	\begin{align}
		\frac{A_\e(\rho)}{\rho^n}-\frac{A_\e(\sigma)}{\sigma^n}=\int_{-\infty}^{0}\dt \int_{\sigma}^{\rho}\dr \int_{M_t}\frac{n}{r^{n+1}}|H-\na^\perp\psi|^2\zeta_\e(\psi_r)\dm_t\geq 0. \label{mono} 
	\end{align} 
	This shows the monotonicity of $r^{-n}A_\e(r)$ as well as $r^{-n}\mathcal{A}(\mathcal{M}\cap E_r)$ by \eqref{e0}. 

\subsection{Proof of Theorem \ref{thm1}} 
In the proof of the main theorem below, we use the following: 
\begin{lemma} \label{pflem}
	For any fixed $t<0$, we have 
	\begin{align*}
	\int_{0}^{\infty} \frac{n}{r^{n+1}}\zeta_\e(\psi_r)\dr = e^{\alpha(\eta_\e)}\Phi, 
	\end{align*}
	where 
	\[e^{\a(\eta_\e)}:=\int_{-\infty}^\infty e^{-y}\eta_\e(y)\dy=\int_{0}^\infty e^{-y}\eta_\e(y)\dy.\] 
	Moreover, $\a(\eta_\e)\leq 0$ and $\a(\eta_\e)\to 0$ as $\e\to 0$. 
\end{lemma} 
\begin{proof}
	First, note that 
	\[\ddr\zeta_\e(\psi_r)=\frac{n}{r}\eta_\e(\psi_r). \]
	Using this, integration by parts yields 
	\begin{align*}
		\int_{\sigma}^{\rho}\frac{n}{r^{n+1}}\zeta_\e(\psi_r)dr=-\left[\frac{1}{r^n}\zeta_\e(\psi_r)\right]_{\sigma}^{\rho}+\int_{\sigma}^{\rho}\frac{n}{r^{n+1}}\eta_\e(\psi_r)dr. 
	\end{align*} 
	The second term above becomes 
	\[\int_{\sigma}^{\rho}\frac{n}{r^{n+1}}\eta_\e(\psi_r)dr=\Phi\int_{\psi_{\sigma}}^{\psi_{\rho}}e^{-y}\eta_\e(y)dy\] 
	by changing variables, $\psi_r=y$. Since 
	\begin{align*}
		\psi_{\sigma} \to -\infty, \quad  \psi_{\rho} \to \infty, \quad \left[\frac{1}{r^n}\zeta_\e(\psi_r)\right]_{\sigma}^{\rho}\to 0 \quad (\text{as $\sigma\to 0, \rho\to \infty$}), 
	\end{align*} 
	we get  
	\[\int_{0}^{\infty} \frac{n}{r^{n+1}}\zeta_\e(\psi_r)\dr = e^{\a(\eta_\e)}\Phi. \] 
	
	For any $\e>0$, we have 
	\[e^{\a(\eta_\e)}=\int_{-\infty}^\infty e^{-y}\eta_\e(y)\dy=\int_{0}^\infty e^{-y}\eta_\e(y)\dy\leq \int_0^\infty e^{-y}dy=1.\]  
	Hence $\a(\eta_\e)\leq 0$. Furthermore, the dominated convergence theorem implies  
	\[\lim_{\e\to 0}e^{\a(\eta_\e)}=\lim_{\e\to 0}\int_{-\infty}^\infty e^{-y}\eta_\e(y)\dy=\int_{-\infty}^\infty e^{-y}\lim_{\e\to 0}\eta_\e(y)\dy=\int_{-\infty}^\infty e^{-y}\delta(y)dy=e^{0}=1.\] 
	\end{proof} 
	
The following lemma confirms that $A_{\e}(r)$ is exactly a smooth approximation of $\mathcal{A}(\mathcal{M} \cap E_r)$. 	
	\begin{lemma}
	\begin{align}
		\lim_{\e\to 0}A_\e(r)=\mathcal{A}\left(\mathcal{M}\cap E_{r}\right). \label{e0}
	\end{align} 
	Moreover,  
	\begin{align}
	\lim_{\e\to 0} \lim_{r\to 0} \frac{A_\e(r)}{r^n}=\Theta(\mathcal{M}, 0, 0), \quad \lim_{\e\to 0} \lim_{r\to \infty} \frac{A_\e(r)}{r^n}=\lim_{r\to \infty }\frac{\mathcal{A}\left(\mathcal{M}\cap E_{r}\right)}{r^n}. \label{e0r0infty}
	\end{align} 
	\end{lemma}
	\begin{proof}
		Note that 
		\[E_{e^{-\frac{\e}{n}}r}=\{\psi_r-\e>0\}. \] 
		Using this with \eqref{zeta} and \eqref{Z}, it is not difficult to check that 
		\begin{align*}
		e^{-\e}\frac{\mathcal{A}\left(\mathcal{M}\cap E_{e^{-\frac{\e}{n}}r}\right)}{r^n}\leq \frac{A_\e(r)}{r^n}\leq  \frac{\mathcal{A}\left(\mathcal{M}\cap E_{r}\right)}{r^n} \quad \text{for all $r>0$. }
	\end{align*}
		This immediately implies the lemma. Here we use \eqref{at0} for the first equality in \eqref{e0r0infty}. 
	\end{proof}

Now we are ready to prove Theorem \ref{thm1}. By the monotonicity, $\lim_{r\to \infty} r^{-n} \mathcal{A}(\mathcal{M}\cap E_r)$ exists but may be infinity. If $\lim_{r\to \infty} r^{-n} \mathcal{A}(\mathcal{M}\cap E_r)=\infty$, then it follows from \eqref{Ec-vs-Hui} or Remark \ref{Ec<Hui} that $\Theta(\mathcal{M},\infty)=\lim_{t\to-\infty}\int_{M_t}\Phi d\mu_t=\infty$, and the claim of Theorem \ref{thm1} trivially holds. Therefore, we may only consider the case that $\lim_{r\to \infty} r^{-n} \mathcal{A}(\mathcal{M}\cap E_r)<\infty$. 
\begin{proof}[Proof of the Main Theorem]
	Letting $\sigma\to 0$, $\rho\to \infty$ in \eqref{mono} with Lemma \ref{pflem} and Fubini's theorem, we have 
	\begin{align*}
		\lim_{\rho \to \infty} \frac{A_\e(\rho)}{\rho^n}- \lim_{\sigma\to 0} \frac{A_\e(\sigma)}{\sigma^n}&=\int_{-\infty}^{0}\dt \int_{0}^{\infty}\dr \int_{M_t}\frac{n}{r^{n+1}}|H-\na^\perp\psi|^2\zeta_\e(\psi_r)\dm_t \\
		&= e^{\a(\eta_\e)}\int_{-\infty}^0\dt \int_{M_t}|H-\na^\perp \psi|^2\Phi \dm_t \\
		&=e^{\a(\eta_\e)} \iint_{\mathcal{M}}|H-\na^\perp \psi|^2\Phi.
	\end{align*}
	Note that the both sides are finite since we assume $\lim_{r\to \infty} r^{-n} \mathcal{A}(\mathcal{M}\cap E_r)<\infty$. Taking $\e\to 0$ in this equality, \eqref{e0r0infty} leads to 
	\begin{align}
		\lim_{r\to \infty}\frac{\mathcal{A}(\mathcal{M}\cap E_r)}{r^n}-\Theta(\mathcal{M}, 0, 0)=\iint_{\mathcal{M}} |H-\na^\perp\psi|^2 \Phi. \label{local}
	\end{align}

	On the other hand, integrating Huisken's monotonicity formula \eqref{Huisken} with respect to $t$ in $-\infty<\tau\leq t\leq \theta<0$, we have 
	\[\int_{M_{\theta}}\Phi\dm_{\theta} -\int_{M_{\tau}}\Phi\dm_{\tau}=-\int_{\tau}^{\theta}\dt \int_{M_t}|H-\na^\perp\psi|^2\Phi \dm_t. \]
	Then we take $\tau \to -\infty $ and $\theta \to 0$ to obtain 
	\begin{align}
		\lim_{t\to -\infty}\int_{M_t}\Phi\dm_t-\Theta(\mathcal{M}, 0, 0)=\iint_{\mathcal{M}} |H-\na^\perp\psi|^2\Phi. \label{global}
	\end{align}
	Comparing \eqref{local} with \eqref{global}, we have the desired relation 
	\[\lim_{r\to\infty}\frac{\mathcal{A}(\mathcal{M}\cap E_r)}{r^n}=\lim_{t\to -\infty}\int_{M_t}\Phi\dm_t=\Theta(\mathcal{M}, \infty). \]
\end{proof}

\section{Applications} 

In this section we will see that Ecker's integral coincides with the entropy at infinity for ancient mean curvature flows. We always assume that $\mathcal{M}=(M_t)_{t<0}$ is an ancient mean curvature flow with \eqref{finite}, that is, Huisken's integral on each time-slice is finite.   
 \subsection{Entropy} 
Choose any point $(x_0, t_0)\in \R^N\times (0, \infty)$ and fix it. The \textit{$F$-functional} for $M_t\subset \R^N$ is defined by 
\[F_{x_0, t_0}(M_t):=\frac{1}{(4\pi t_0)^{\frac{n}{2}}}\int_{M_t}e^{-\frac{|x-x_0|^2}{4t_0}}\dm_t=\int_{M_t}\Phi_{x_0, t_0}(x, 0)\dm_t, \]
where $\Phi_{x_0, t_0}(x, t)=\Phi(x-x_0, t-t_0)$ (see Remark \ref{rmk_mono}). Then the \textit{entropy} of $M_t$ is defined by the supremum of the $F$-functional: 
 \[\lambda(M_t):=\sup_{x_0, t_0}F_{x_0, t_0}(M_t). \] 
In the following argument, we use some simple properties of the entropy: 
\begin{enumerate}
	\item $\lambda({M_t})$ is invariant under translations and scalings in $\R^N$, i.e.,   
	\[\lambda(cM_t+v)=\lambda(cM_t)=\lambda(M_t), \quad v\in \R^N, c>0.\] 
		\item $\lambda(M_t)$ is nonincreasing along mean curvature flow, i.e.,   
	\[\lambda(M_s)\geq \lambda(M_t), \quad s<t. \]
\end{enumerate}
Note that the property (2) follows from Huisken's monotonicity formula (see Remark \ref{rmk_mono}). 
 
\begin{lemma} \label{ent}
	For an ancient mean curvature flow $\mathcal{M}=(M_t)_{t<0}$ with \eqref{finite}, we have
	\[\lim_{t\to -\infty}\lambda(M_t)=\sup_{t<0}\lambda(M_t)=\Theta(\mathcal{M}, \infty). \]
\end{lemma}
\begin{proof}
	From the definition of the $F$-functional and the entropy, we compute 
	\begin{align*}
		\int_{M_t}\Phi_{x_0, t_0}(x, t)\dm_t =\int_{M_t}\Phi_{x_0, t_0-t}(x, 0)\dm_t=F_{x_0, t_0-t}(M_t) \leq \lambda(M_t).  
	\end{align*}
	Letting $t\to -\infty$ in this inequality, we have 
	\[\Theta(\mathcal{M}, \infty)=\lim_{t\to -\infty}\int_{M_t}\Phi_{x_0, t_0}(x, t)\dm_t\leq \lim_{t\to -\infty} \lambda(M_t).\] 
	Here, we used the fact that the quantity 
	\[\lim_{t\to -\infty}\int_{M_t}\Phi_{x_0, t_0}(x, t)\dm_t\] 
	does not depend on the choice of the point $(x_0, t_0)\in \R^N\times (0, \infty)$ (see Remark \ref{rmk_mono}). 
	
	On the other hand, it holds from Huisken's monotonicity formula that 
	\begin{align*}
		F_{x_0, t_0}(M_t)&=\int_{M_t}\Phi_{x_0, t_0}(x, 0) \dm_t\\
		&=\int_{M_t}\Phi_{x_0, t_0+t}(x, t) \dm_t \leq \int_{M_s}\Phi_{x_0, t_0+t}(x, s) \dm_s\leq \Theta(\mathcal{M}, \infty)
	\end{align*}
	for $s<t$. Taking supremum with respect to $(x_0, t_0)$ in this inequality, we have $\lambda(M_t)\leq \Theta(\mathcal{M}, \infty)$, and hence 
	\[\lim_{t\to -\infty}\lambda(M_t)\leq \Theta(\mathcal{M}, \infty).\] 
	This proves the lemma. 
\end{proof}
As a direct consequence of Theorem \ref{thm1} and Lemma \ref{ent}, we have a relation between Ecker's integral and the entropy. 
\begin{corollary} \label{cor3.2} 
For an ancient mean curvature flow $\mathcal{M}=(M_t)_{t<0}$ with \eqref{finite}, we have 
	\[\lim_{r\to\infty}\frac{\mathcal{A}(\mathcal{M}\cap E_r)}{r^n}=\Theta(\mathcal{M},\infty)=\sup_{t<0}\lambda(M_t). \] 
\end{corollary}

\subsection{Example: entropy and Ecker's integral of translating solitons}
Now we consider \textit{translating solitons} of mean curvature flow. A translating soliton is a solution of mean curvature flow which moves only by a translation in some fixed direction $\mathrm{v} \in \R^N, |\mathrm{v}|=1$. So, a translating soliton is defined for all the time $t\in (-\infty, \infty)$ and can be written as 
\[M_t=M_0+t\mathrm{v}.\]
By the property (1) of the entropy, we have 
\[\lambda(M_t)=\lambda(M_0)\]
for a translating soliton. Therefore we have the following consequence. 
\begin{corollary}
	Let $\mathcal{M}=(M_t)_{t\in \R}=(M_0+t\mathrm{v})_{t\in \R}$ be a translating soliton which satisfies \eqref{finite}. Then we have 
	\[\lim_{r\to\infty}\frac{\mathcal{A}(\mathcal{M}\cap E_r)}{r^n}=\Theta(\mathcal{M},\infty)=\lambda(M_0). \]
\end{corollary} 

Let $\mathcal{G}=(\Gamma_t)_{t\in \R}=(\Gamma_0+t\mathrm{v})_{t\in \R}$ be the \textit{grim reaper} and $\mathcal{B}=(B_t)_{t\in \R}=(B_0+t\mathrm{v})_{t\in \R}$ be the \textit{bowl soliton}. The grim reaper is a plane curve translating soliton, and the bowl soliton is an $n$-dimensional hypersurface which is rotationally symmetric, strictly convex, and entire graphic. In \cite{Gu19}, we can find explicit computations of the entropy for $\mathcal{G}$ and $\mathcal{B}$ by Guang: 
\[\lambda(\Gamma_0)=2, \quad \lambda(B_0)=\lambda(S^{n-1})=n\omega_n \left(\frac{n-1}{2\pi e}\right)^{\frac{n-1}{2}}, \] 
where $S^{n-1}\subset \R^n$ is the standard $(n-1)$-dimensional sphere and $\omega_n$ is the volume of unit ball in $\R^n$.


\end{document}